\documentclass[letterpaper,11pt,twoside,keywordsasfootnote,addressatend,noinfoline]{article}
\usepackage{fullpage}
\usepackage[english]{babel}
\usepackage{amssymb}
\usepackage{amsmath}
\usepackage{theorem}
\usepackage{epsfig}
\usepackage{subfigure}
\usepackage{lscape}
\usepackage{imsart}

\theorembodyfont{\slshape}
\newtheorem{theor}{Theorem}

\newtheorem{lemma}[theor]{Lemma}

\newenvironment{proof}{\noindent{\scshape Proof.}}{\hspace{2mm} $\square$}

\newcommand{\N}{\mathbb{N}}
\newcommand{\Z}{\mathbb{Z}}

\newcommand{\ind}{\hbox{{\small 1} \hspace*{-11pt} 1}}
\newcommand{\ep}{\epsilon}

\DeclareMathOperator{\card}{card \,}

\begin{document}

\begin{frontmatter}

\title     {The role of dispersal in interacting patches \\ subject to an Allee effect}
\runtitle  {The role of dispersal in interacting patches    subject to an Allee effect}
\author    {N. Lanchier\thanks{Research supported in part by NSF Grant DMS-10-05282.}}
\runauthor {N. Lanchier}
\address   {School of Mathematical and Statistical Sciences, \\ Arizona State University, \\ Tempe, AZ 85287, USA.}

\begin{abstract} \ \
 This article is concerned with a stochastic multi-patch model in which each local population is subject to a strong Allee effect.
 The model is obtained by using the framework of interacting particle systems to extend a stochastic two-patch model that has been
 recently introduced by Kang and the author.
 The main objective is to understand the effect of the geometry of the network of interactions, which represents potential
 migrations between patches, on the long-term behavior of the metapopulation.
 In the limit as the number of patches tends to infinity, there is a critical value for the Allee threshold below which the
 metapopulation expands and above which the metapopulation goes extinct.
 Spatial simulations on large regular graphs suggest that this critical value strongly depends on the initial distribution when
 the degree of the network is large whereas the critical value does not depend on the initial distribution when the degree is small.
 Looking at the system starting with a single occupied patch on the complete graph and on the ring, we prove analytical results
 that support this conjecture.
 From an ecological perspective, these results indicate that, upon arrival of an alien species subject to a strong Allee effect to
 a new area, though dispersal is necessary for its expansion, strong long range dispersal drives the population toward extinction.
\end{abstract}



\end{frontmatter}


\section{Introduction}
\label{sec:intro}

\indent To understand the role of dispersal in populations subject to a strong Allee effect, Kang and the author recently introduced
 deterministic and stochastic two-patch models \cite{kang_lanchier_2011}.
 In population dynamics, the term Allee effect refers to a certain process that leads to decreasing net population growth with
 decreasing density.
 This monotone relationship may induce the existence of a so-called Allee threshold below which populations are driven toward
 extinction, a phenomenon which is referred to as strong Allee effect.
 In this paper, we continue the analysis initiated in \cite{kang_lanchier_2011} and use the framework of interacting particle
 systems to extend the stochastic two-patch model to a more general multi-patch model.
 Thinking of the set of patches as the vertex set of a graph in which each edge indicates potential dyadic interactions
 between two patches, the main objective is to understand how the geometry of the network affects the survival probability of the
 global metapopulation.
 Let $G = (V, E)$ be a graph representing the network of interactions.
 The system is a continuous-time Markov chain whose state at time $t$ is a function
 $$ \eta_t : V \to [0, 1] \quad \hbox{with} \quad \eta_t (x) = \hbox{population density at vertex} \ x. $$
 Having an Allee threshold $\theta \in (0, 1)$ and a migration factor $\mu \in (0, 1/2]$, the evolution consists of the following
 two elementary events occurring in continuous time:
\begin{list}{\labelitemi}{\leftmargin=1.5em}
 \item {\bf Mixing events} -- Each edge becomes active at rate one, which results in a fraction $\mu$ of the population at each
  of the two interacting vertices to move to the other vertex. \vspace*{4pt}
 \item {\bf Local events} -- Each vertex becomes active at rate one, which results in the population density at that vertex to
  jump from below $\theta$ to state 0 or from above $\theta$ to state 1.
  If the density is equal to the Allee threshold then it jumps to either state 0 or state 1 with probability 1/2.
\end{list}
 The inclusion of mixing events indicates that individuals can move from patch to patch through the edges of the graph.
 The strength of dispersal is thus modeled by the mean degree distribution of the network of interactions.
 Local events model the presence of a strong Allee effect in each patch:
 local populations below the Allee threshold are driven toward extinction whereas local populations above the Allee threshold
 expand, and we think of state 1 as the normalized density of a local population at carrying capacity.
 Formally, the Markov generator is given by
\begin{equation}
\label{eq:generator-1}
  \begin{array}{l}
   L_{\eta} f (\eta) \ = \ \displaystyle \sum_{(x, y) \in E} \ \ [f (\sigma_{x, y} \,\eta) - f (\eta)] \\ \hspace{80pt} + \
                           \displaystyle \sum_{x \in V} \ (\ind \{\eta (x) > \theta \} + (1/2) \,\ind \{\eta (x) = \theta \}) \ [f (\sigma_x^+ \,\eta) - f (\eta)] \\ \hspace{120pt} + \
                           \displaystyle \sum_{x \in V} \ (\ind \{\eta (x) < \theta \} + (1/2) \,\ind \{\eta (x) = \theta \}) \ [f (\sigma_x^- \,\eta) - f (\eta)] \end{array}
\end{equation}
 where $(\sigma_x^+ \,\eta) (x) = 1$ and $(\sigma_x^- \,\eta) (x) = 0$, and where
 $$ (\sigma_{x, y} \,\eta) (z_1) \ = \ \eta (z_1) + \mu \,(\eta (z_2) - \eta (z_1)) \ \ \hbox{whenever} \ \ \{z_1, z_2 \} = \{x, y \} $$
 while the state at all other vertices is unchanged.
 The stochastic two-patch model we introduced and studied in \cite{kang_lanchier_2011} is simply the process \eqref{eq:generator-1}
 when the network of interactions consists of two vertices connected by a single edge.
 The main objective was to answer the following question:
 starting with one empty patch and one patch at carrying capacity, does the inclusion of mixing events lead to a global
 extinction, i.e., both patches in state 0 eventually, or to a global expansion, i.e., both patches in state 1 eventually?
 Theorem 8 in \cite{kang_lanchier_2011} gives the following answer: when the migration factor is small, the probability of
 global extinction, respectively, global expansion, is close to one when the Allee threshold is larger than one half,
 respectively, smaller than one half.
 As the migration factor increases, the limit becomes less predictable.
 This result suggests that one half is a critical value for the Allee threshold.
 However, we literally interpreted this one half as one patch initially in state 1 divided by two patches, and also conjectured that,
 for systems in which all the patches are connected, i.e., the network of interactions is a complete graph, the critical value
 for the Allee threshold is equal to the initial fraction of patches in state 1 in the limit as the size of the system tends to
 infinity.
 This conjecture is supported by the last diagram of Figure \ref{fig:regular} which shows numerical results for the stochastic
 process \eqref{eq:generator-1} on four regular graphs with different degrees.
 In contrast, the first diagram suggests that the critical value for the Allee threshold is again equal to one half for the process
 on the ring in the limit as the size of the system tends to infinity.
 This indicates that the critical value for the Allee threshold strongly depends on the initial configuration in the presence of
 strong long range dispersal whereas the initial configuration is essentially unimportant in the presence of weak short range
 dispersal.
 The main objective of this paper is to explain at least qualitatively this difference between the process on the complete graph,
 which models strong dispersal, and the process on the ring, which models weak dispersal.

\begin{figure}[t!]
 \centering
\scalebox{0.40}{\input{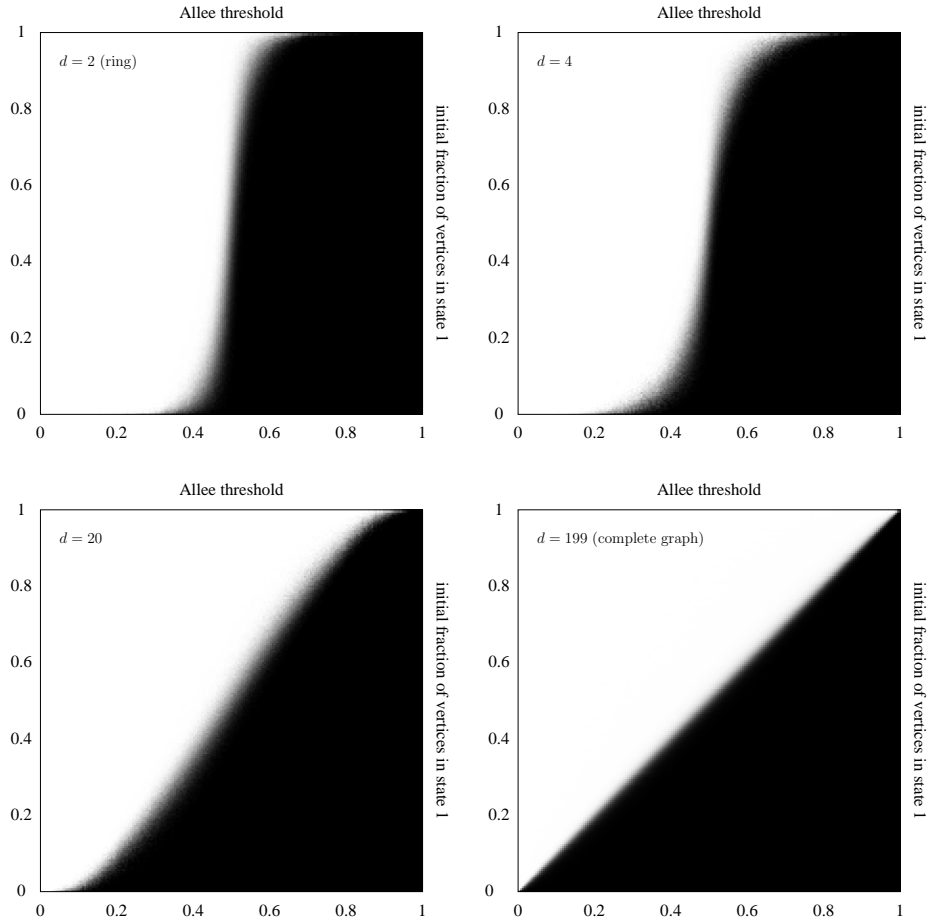}}
 \caption{\upshape Simulation results for the process on the torus $\Z / 200 \Z$ in which each vertex is connected to its $d$
  nearest neighbors.
  Each picture shows the density of vertices below/above the Allee threshold after a large number of updates as a function of
  the Allee threshold (200 values) and the initial fraction of vertices in state 1 (200 values).
  The color for each of the $200 \times 200$ parameter values is computed from the average of 100 independent realizations, and
  the color code is black for all below the Allee threshold and white for all above.
  In all the simulations, $\mu = 0.2$.}
\label{fig:regular}
\end{figure}

\pagebreak


\section{Main results}
\label{sec:results}

\indent We first assume that the process starts from the product measure in which each vertex is in state~1 with probability
 $\rho$ and in state 0 otherwise.
 Following the terminology of \cite{kang_lanchier_2011}, we call global extinction the event that the process converges to the
 ``all 0'' configuration and global expansion the event that it converges to the ``all 1'' configuration.
 Note that, on finite graphs, the process converges to one of these two absorbing states.
 In particular, the probability of global extinction and the probability of global expansion sum up to one, thus showing that
 the long-term behavior is completely characterized by the probability of global expansion
 $$ p_G (\theta, \mu, \rho) \ := \ P \,(\eta_t \equiv 1 \ \hbox{for some} \ t), $$
 which depends on the network of interactions, the Allee threshold, the migration factor, and the initial density of occupied
 patches.
 The simulation results of Figure \ref{fig:regular} suggest some monotonicity of the probability of expansion with respect to the
 Allee threshold and the initial density of occupied patches, as well as a certain symmetry between the probability of expansion
 and the probability of extinction.
 These results follow directly from standard coupling arguments for interacting particle systems that we briefly describe without
 detailed proof. \vspace*{4pt}

\noindent {\bf Monotonicity with respect to $\theta$} --
 Two processes on the same graph and starting from the same initial configuration but with different Allee thresholds can be
 coupled in such a way that the process with the smaller Allee threshold dominates the other process, which implies that the
 probability of global expansion $p_G (\theta, \mu, \rho)$ is nonincreasing with respect to $\theta$. \vspace*{4pt}

\noindent {\bf Monotonicity with respect to $\rho$} --
 Two processes on the same graph, with the same Allee threshold, and with the same migration factor can be coupled in such a
 way that if one process dominates the other one at time 0 then the domination remains true at all times.
 This implies that the probability of global expansion $p_G (\theta, \mu, \rho)$ is nondecreasing with respect to $\rho$. \vspace*{4pt}

\noindent {\bf Symmetry} --
 The process with Allee threshold $\theta$ can be coupled with the process on the same graph and with the same migration
 factor but with Allee threshold $1 - \theta$ in such a way that if at any vertex the initial population density for one process
 equals one minus the initial population density for the other process then this remains true at all times.
 This implies that
 $$ p_G (\theta, \mu, \rho) \ = \ 1 - p_G (1 - \theta, \mu, 1 - \rho), $$
 and explains the symmetry in the four simulation pictures of Figure \ref{fig:regular}. \vspace*{4pt}

\noindent {\bf The process starting with a single occupied patch} --
 We now return to the main objective of this paper, which is to understand the effect of the geometry of the network on the
 invadability of species subject to a strong Allee effect.
 This aspect is mathematically more difficult to understand because two processes on different graphs cannot be coupled in such
 a way that one process dominates the other one.
 Our analysis focuses on the extreme cases of the ring and the complete graph corresponding to the first and last diagrams
 of Figure \ref{fig:regular}. Let
\begin{equation}
\label{eq:expansion}
 \begin{array}{rcl}
     p_N^+ (\theta, \mu, \rho) & = & \hbox{the probability of global expansion for the process} \\
                               &   & \hbox{on the complete graph with $N$ vertices} \vspace*{4pt} \\
     p_N^- (\theta, \mu, \rho) & = & \hbox{the probability of global expansion for the process} \\
                               &   & \hbox{on the ring with $N$ vertices} \end{array}
\end{equation}
 where the $+$ and $-$ superscripts allude respectively to the fact that the complete graph is the connected regular graph with the
 largest degree while the ring is the connected regular graph with the smallest degree.
 Recall that the simulation results of Figure \ref{fig:regular} suggest that
\begin{equation}
\label{eq:conjecture-complete}
  \begin{array}{rcl}
  \lim_{N \to \infty} \ p_N^+ (\theta, \mu, \rho) \ = \ 0 & \hbox{when} & \theta > \rho  \\
                                                  \ = \ 1 & \hbox{when} & \theta < \rho, \end{array}
\end{equation}
 whereas for the process on the ring starting with $\rho \in (0, 1)$,
\begin{equation}
\label{eq:conjecture-ring}
  \begin{array}{rcl}
  \lim_{N \to \infty} \ p_N^- (\theta, \mu, \rho) \ = \ 0 & \hbox{when} & \theta > 1/2  \\
                                                  \ = \ 1 & \hbox{when} & \theta < 1/2. \end{array}
\end{equation}
 Following \cite{kang_lanchier_2011} whose main objective was to understand whether an alien species established in one
 patch can expand in space, we assume from now on that the process starts with a single patch in state 1 and all
 the other patches in state 0.
 In particular, we drop the parameter $\rho$ in the probabilities \eqref{eq:expansion}.
 The following two theorems give qualitative differences between the system on the complete graph and the system on the ring
 starting with a single vertex in state 1, which supports the conjectures \eqref{eq:conjecture-complete} and \eqref{eq:conjecture-ring}.
 More precisely, our first theorem indicates that, even when the Allee threshold is very small, the system on the complete graph is
 driven toward global extinction with high probability when the number of vertices is large.
\begin{theor} --
\label{thm:complete}
 Assume that $\theta, \mu > 0$.
 Then, $p_N^+ (\theta, \mu) \to 0$ as $N \to \infty$.
\end{theor}
 In contrast, when the Allee threshold is small enough, the system on the ring expands globally with a positive probability
 that does not depend on the number of vertices.
\begin{theor} --
\label{thm:ring}
 Assume that $\theta < \mu^2 \,(1 - \mu)^{1140}$.
 Then, $\inf_N \,p_N^- (\theta, \mu) > 0$.
\end{theor}
 The mysterious assumption in the previous theorem follows from a series of bounds of certain probabilities that are estimated
 based on geometric arguments and are not optimal.
 Some of these estimates appear in our proof and the other ones in the calculation of an upper bound for the critical value of
 one dependent oriented site percolation in, e.g., \cite{durrett_1984}.
 Even though the assumption of the theorem is far from being optimal, it gives at least an explicit lower bound for the critical
 value of the Allee threshold for the process on the ring.
 More importantly, the combination of both theorems show the following qualitative difference:
 for some values of the Allee threshold, the probability of global expansion is bounded from below for the process on the ring
 but vanishes to zero for the process on the complete graph as the number of vertices increases.
 This supports at least qualitatively the contrast between \eqref{eq:conjecture-complete} and \eqref{eq:conjecture-ring}.
 From an ecological perspective, this indicates that, upon arrival of an alien species to a new area, though dispersal is
 necessary for its expansion, the best strategy is to first only disperse to nearby patches, and then progressively increase
 the strength of its dispersal as the fraction of patches at carrying capacity increases.


\section{Preliminary results}
\label{sec:preliminary}

\indent This section gives some definitions and simple results that will be used repeatedly in the proof of both theorems.
 Throughout this paper, we think of the process \eqref{eq:generator-1} as being constructed from a Harris' graphical
 representation \cite{harris_1972}.
 Each edge of the graph is equipped with a Poisson process with intensity one while each vertex is equipped with a Poisson process
 with intensity one and a sequence of independent Bernoulli random variables with parameter 1/2. We write
\begin{itemize}
 \item $T_n (x, y)$ := $n$th arrival time of the Poisson process attached to edge $(x, y) \in E$, \vspace*{4pt}
 \item $U_n (x)$ := $n$th arrival time of the Poisson process attached to vertex $x \in V$, \vspace*{4pt}
 \item $B_n (x)$ := $n$th member of the Bernoulli sequence attached to vertex $x \in V$.
\end{itemize}
 All these Poisson processes and Bernoulli random variables are independent and form together a percolation structure from which
 the multi-patch model \eqref{eq:generator-1} can be constructed.
\begin{itemize}
 \item {\bf Mixing events} -- At time $t := T_n (x, y)$, we draw a double arrow along the corresponding edge to indicate
  the occurrence of the following mixing event:
  $$ \eta_t (x) \ = \ (\sigma_{x, y} \,\eta_{t-}) (x) \quad \hbox{and} \quad \eta_t (y) \ = \ (\sigma_{x, y} \,\eta_{t-}) (y). $$
 \item {\bf Local events} -- At time $t := U_n (x)$, we put a dot at vertex $x$ to indicate that
  $$ \eta_t (x) \ = \ \ind \{\eta_{t-} (x) > \theta \} \ + \ B_n (x) \ \ind \{\eta_{t-} (x) = \theta \}. $$
\end{itemize}
 In the proof of both theorems, we first study the process $(\xi_t)$ that includes mixing events but excludes local
 events whose dynamics is therefore described by the Markov generator
\begin{equation}
\label{eq:generator-2}
  L_{\xi} f (\xi) \ = \sum_{(x, y) \in E} \ \ [f (\sigma_{x, y} \,\xi) - f (\xi)].
\end{equation}
 Note that this process can be constructed graphically as previously by only using the Poisson processes attached to the edges of the graph.
 Note also that, since the state at each vertex is a convex combination of the states of the vertices at earlier times,
 $$ \xi_s (x) > \theta \ \ \hbox{for all} \ x \in V \quad \hbox{implies that} \quad \xi_t (x) > \theta \ \ \hbox{for all} \ (x, t) \in V \times (s, \infty) $$
 and the analogous implication obtained by flipping the inequalities.
 Following the terminology introduced in \cite{kang_lanchier_2011}, we call respectively upper/lower configurations the sets
 $$ \begin{array}{rcl}
    \Omega_+ & = & \{\xi : V \to [0, 1] \ \hbox{such that} \ \xi (x) > \theta \ \hbox{for all} \ x \in V \} \vspace*{4pt} \\
    \Omega_- & = & \{\xi : V \to [0, 1] \ \hbox{such that} \ \xi (x) < \theta \ \hbox{for all} \ x \in V \} \end{array} $$
 and observe that the previous implication means that, once the process \eqref{eq:generator-2} hits the set of upper configurations,
 it stays in this set forever.
 By definition of the Allee threshold, the same holds for the original process \eqref{eq:generator-1} from which we deduce the
 following lemma.
\begin{lemma} --
\label{lem:upper-lower}
 For the process \eqref{eq:generator-1} on a finite graph,
 $$ \begin{array}{rcl}
    \eta_t \equiv 1 \ \ \hbox{for some} \ t & \hbox{if and only if} & \eta_t \in \Omega_+ \ \hbox{for some} \ t \vspace*{2pt} \\
    \eta_t \equiv 0 \ \ \hbox{for some} \ t & \hbox{if and only if} & \eta_t \in \Omega_- \ \hbox{for some} \ t.
    \end{array} $$
\end{lemma}
 Lemma \ref{lem:upper-lower} is one of the keys to proving both theorems.
 According to the lemma, it suffices to prove that, with the appropriate probability, the process on the complete graph hits a
 lower configuration whereas the process on the ring hits an upper configuration.


\section{The process on the complete graph}
\label{sec:complete}

\indent This section is devoted to the proof of Theorem \ref{thm:complete}.
 As previously mentioned, the first step is to study the process \eqref{eq:generator-2} that excludes local events but includes
 mixing events.
 It is obvious that, when starting with a single vertex in state 1, this process eventually hits a lower configuration provided
 the number of vertices is sufficiently large.
 The key to the proof is to show that the time to hit a lower configuration can be made arbitrarily small, which relies
 on large deviation estimates for several random variables that we now define.
 We call collision a mixing event that involve two occupied vertices, i.e., a collision occurs at time $t$ whenever
 $$ t \ = \ T_n (x, y) \quad \hbox{and} \quad \min (\xi_{t-} (x), \xi_{t-} (y)) \neq 0. $$
 Then, we define the three random variables
 $$ \begin{array}{rcl}
    \hbox{time to dispersion} & : & \tau_D \ := \ \inf \,\{t : \xi_t \in \Omega_- \} \vspace*{4pt} \\
    \hbox{time to collision}  & : & \tau_C \ := \ \hbox{first time a collision occurs} \vspace*{4pt} \\
    \hbox{number of occupied patches} & : & |\xi_t| \ := \ \card \{x \in \Z / N \Z : \xi_t (x) \neq 0 \}. \end{array} $$
 We will prove that, with probability close to one when $N$ is large,
\begin{equation}
\label{eq:events}
  \tau_D \ \leq \ T_N \quad \hbox{and} \quad \tau_C \ > \ T_N \quad \hbox{and} \quad |\xi_{T_N}| \ \leq \ 4^{N T_N}
\end{equation}
 for some $T_N$ that tends to zero as $N \to \infty$.
 The probability of the events in \eqref{eq:events} will be estimated backwards by conditioning, i.e., the probability
 involving the time to collision is obtained by conditioning on the number of occupied patches, while the probability
 involving the time to dispersion is obtained by conditioning on the time to collision.
 To complete the proof, we will return to the process \eqref{eq:generator-1} and use our estimates for the
 probability of the first and last events in \eqref{eq:events} to prove that the probability that a local event occurs
 in any of the occupied patches before the time to dispersion tends to zero as the number of vertices $N \to \infty$. \vspace*{4pt}

\noindent {\bf Mapping to a dynamic graph} --
 To estimate the probability of the events in \eqref{eq:events}, we first define a mapping to visualize the evolution of
 the process \eqref{eq:generator-2} through a dynamic graph, i.e., a continuous-time Markov chain whose state at time $t$
 is a random oriented graph
 $$ H_t \ := \ (V_t, E_t) \quad \hbox{where} \quad V_t \ \subset \ V \times \N. $$
 The dynamic graph is coupled with the process \eqref{eq:generator-2} and defined as follows.
\begin{itemize}
 \item The graph $H_0$ has only one vertex, namely $(x_0, 0) \in V \times \N$ where $x_0$ is the single vertex in state 1
  initially, and no (oriented) edge. \vspace*{4pt}
 \item We call $(x, i)$ a leaf at time $t$ whenever $(x, i) \in V_t$ and $(x, i + 1) \notin V_t$. \vspace*{4pt}
 \item Assume that $(x, i) \in V_{t-}$ is a leaf and that $t = T_n (x, x')$ for some $n \geq 1$.
  Then, we define the new vertex set and the new edge set as
 \begin{equation}
 \label{eq:growth}
   \begin{array}{rcl}
    V_t & := & V_{t-} \,\cup \,\{(x, i + 1), (x', i + 1) \} \vspace{4pt} \\
    E_t & := & E_{t-} \,\cup \,\{(x, i) \to (x, i + 1), (x, i) \to (x', i + 1) \}. \end{array}
 \end{equation}
\end{itemize}
 The left-hand side of Figure \ref{fig:tree} gives a schematic picture of a realization of the process with only mixing
 events where $x_j$ denotes the $j$th vertex that becomes occupied, while the right-hand side of the figure gives a picture
 of the corresponding graph.
 Note that a vertex has a positive density at time $t$ if and only if it is the first coordinate of a leaf of $H_t$.
 It follows that
\begin{equation}
\label{eq:leaves}
 |\xi_t| \ = \ \hbox{number of leaves in} \ H_t
\end{equation}
 and, using in addition the evolution rule \eqref{eq:growth}, that a collision event results in two pairs of oriented edges
 each pointing to the same leaf.
 In particular, 
\begin{equation}
\label{eq:tree}
 \hbox{$H_t$ is an oriented binary tree with root $(x_0, 0)$ if and only if $t < \tau_C$}.
\end{equation}
 From properties \eqref{eq:growth} and \eqref{eq:tree}, it also follows that
\begin{equation}
\label{eq:path}
 (x, i) \ \hbox{is a leaf at time} \ t < \tau_C \quad \hbox{implies that} \quad \xi_t (x) \ \leq \ (1 - \mu)^i.
\end{equation}
 In fact, if $(x, i)$ is a leaf at time $t < \tau_C$ then there exists a unique oriented path from the root to this leaf
 and the density at vertex $x$ can be computed explicitly looking at the number of vertical edges in this path, but this
 property is not needed in the proof of the theorem. \vspace*{4pt}

\begin{figure}[t]
\centering
\scalebox{0.40}{\input{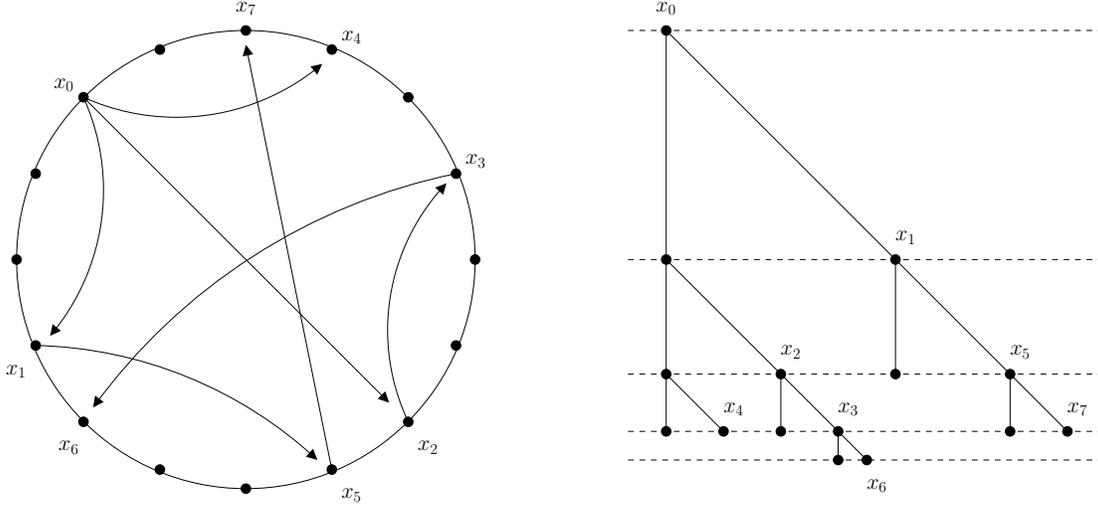}}
\caption{\upshape{Pictures related to the proof of Theorem \ref{thm:complete}}}
\label{fig:tree}
\end{figure}

\noindent {\bf Time to dispersion} --
 The next step is to use \eqref{eq:leaves}--\eqref{eq:path} and the dynamic graph representation of the process to
 estimate the probability of the three events in \eqref{eq:events}.
 Under the assumptions of the theorem, there exists $n$ such that $(1 - \mu)^n < \theta$.
 We then define
 $$ T_N \ := \ n \,\ln \,(\ln N) / N \qquad \hbox{and} \qquad K_N \ := \ 4^{n \ln \,(\ln N)} \ = \ 4^{N T_N}. $$
 The following three lemmas give estimates of the probability of the last event, the second event and the first event
 in \eqref{eq:events}, respectively, for the deterministic time $T_N$ defined above.
\begin{lemma} --
\label{lem:leaves}
 There is $a > 0$ such that $P \,(|\xi_{T_N}| > K_N) \leq (\ln N)^{- a}$ for all $N$ large.
\end{lemma}
\begin{proof}
 The number of leaves is maximal when there is no collision, in which case the number of leaves jumps from $i$ to $i + 1$ at rate $iN$.
 This together with \eqref{eq:leaves} implies that
 $$ E \,|\xi_{T_N}| \ = \ E \,(\hbox{number of leaves in} \ H_{T_N}) \ \leq \ 2^{N T_N} \ = \ \sqrt{K_N}. $$
 In particular, large deviation estimates for the Poisson distribution give
 $$ P \,(|\xi_{T_N}| > K_N) \ = \ P \,(|\xi_{T_N}| > 4^{N T_N}) \ \leq \ \exp (- a \ln \,(\ln N)) \ = \ (\ln N)^{- a} $$
 for a suitable constant $a > 0$ and all $N$ sufficiently large.
\end{proof}
\begin{lemma} --
\label{lem:collision}
 For all $N$ large,  $P \,(\tau_C \leq T_N) \leq 2 \,(\ln N)^{-a}$.
\end{lemma}
\begin{proof}
 Given that the graph $H_t$ has $i$ leaves, the probability of a collision at the next update of the system is equal to $i / N$.
 In particular, the conditional probability of a collision before $T_N$ given that the number of leaves at that time is smaller
 than $K_N$ is
 $$ \begin{array}{rcl}
    \displaystyle P \,(\tau_C \leq T_N \ | \,| \,\xi_{T_N}| \leq K_N) \ \leq \
    \displaystyle \sum_{i = 1}^{K_N} \ \frac{i}{N} & \leq &
    \displaystyle \frac{K_N (K_N + 1)}{2N} \\ & \leq &
    \displaystyle \frac{\exp (4n \ln \,(\ln N))}{2N} \ = \ \frac{(\ln N)^{4n}}{2N}. \end{array} $$
 This, together with Lemma \ref{lem:leaves}, implies that
 $$ \begin{array}{rcl}
  P \,(\tau_C \leq T_N) & \leq & P \,(\tau_C \leq T_N \,| \,|\xi_{T_N}| \leq K_N) \ + \ P \,(|\xi_{T_N}| > K_N) \vspace*{4pt} \\
                     & \leq & (\ln N)^{4n} / 2N \ + \ (\ln N)^{-a} \ \leq \ 2 \,(\ln N)^{-a} \end{array} $$
 for all $N$ sufficiently large.
\end{proof}
\begin{lemma} --
\label{lem:dispersion}
 Let $a > 0$ as in Lemmas \ref{lem:leaves} and \ref{lem:collision}.
 Then, for all $N$ large,
 $$ P \,(\tau_D > T_N) \ \leq \ 2^{n + 1} \,(\ln N)^{-1} \ + \ 2 \,(\ln N)^{-a}. $$
\end{lemma}
\begin{proof}
 Motivated by \eqref{eq:path}, we introduce the stopping times
 $$ \sigma_j \ := \ \inf \,\{t : \hbox{for each leaf} \ (x, i) \in V_t \ \hbox{we have} \ i \geq j \}. $$
 The first step is to prove that $\sigma_n \leq T_N$ with probability arbitrarily close to one when the number of vertices is large.
 Note that, according to the evolution rules \eqref{eq:growth}, we have
\begin{equation}
\label{eq:dispersion-1}
  \card \{x \in V : (x, i) \in V_t \} \ \leq \ 2^i \quad \hbox{for all} \ t \geq 0.
\end{equation}
 Moreover, since each vertex has degree $N - 1$ and is therefore involved in a mixing event at the arrival times of a Poisson process
 with intensity $N - 1$, we have
\begin{equation}
\label{eq:dispersion-2}
  P \,(\inf \,\{t : (x, i + 1) \in V_t \} - \inf \,\{t : (x, i) \in V_t \} > T) \ = \ \exp (- (N - 1) \,T)
\end{equation}
 for all $x \in V$ such that $(x, i) \in V_t$ at some time $t$, i.e., the amount of time a vertex in the dynamic graph is a leaf is
 exponential with parameter $N - 1$.
 From \eqref{eq:dispersion-1} and \eqref{eq:dispersion-2}, we deduce that the temporal increment $\sigma_{j + 1} - \sigma_j$ required
 to grow one more generation in the dynamic graph is stochastically smaller than the maximum of $2^j$ independent exponential random
 variables with the same parameter $N - 1$.
 In particular, having a collection $\ep_1, \ep_2, \ep_3, \ldots$ of independent exponential random variables with
 parameter $N - 1$, we deduce that
\begin{equation}
\label{eq:dispersion-3}
  \begin{array}{l}
   \displaystyle P \,(\sigma_n > T_N) \ \leq \
   \displaystyle \sum_{j = 0}^{n - 1} \ P \,(\sigma_{j + 1} - \sigma_j > T_N / n) \vspace*{-4pt} \\ \hspace*{20pt} \leq \
   \displaystyle \sum_{j = 0}^{n - 1} \ P \,(\max \,\{\ep_i : i \leq 2^j \} > T_N / n) \ \leq \
   \displaystyle \sum_{j = 0}^{n - 1} \ P \,(\ep_i > T_N / n \ \hbox{for some} \ i \leq 2^j) \vspace*{-4pt} \\ \hspace*{50pt} \leq \
   \displaystyle \sum_{j = 0}^{n - 1} \ 2^j \,\exp (- (N - 1) \,\ln (\ln N) / N) \ \leq \
   \displaystyle 2^{n + 1} \,(\ln N)^{-1} \end{array}
\end{equation}
 for all $N$ large.
 In other respects, using \eqref{eq:path} and recalling the definition of $n$,
\begin{equation}
\label{eq:dispersion-4}
 \begin{array}{rcl}
  \sigma_n \leq T_N < \tau_C & \hbox{implies that} &
  \xi_{T_N} (x) \leq (1 - \mu)^n < \theta \ \hbox{for each leaf} \ (x, i) \in V_{T_N} \vspace*{4pt} \\ & \hbox{implies that} &
  \tau_D \leq T_N. \end{array}
\end{equation}
 Combining \eqref{eq:dispersion-3} and \eqref{eq:dispersion-4}, and using Lemma \ref{lem:collision}, we conclude that
 $$ P \,(\tau_D > T_N) \ \leq \ P \,(\sigma_N > T_N) \ + \ P \,(\tau_C \leq T_N) \ \leq \ 2^{n + 1} \,(\ln N)^{-1} \ + \ 2 \,(\ln N)^{-a} $$
 which completes the proof.
\end{proof} \\ \\
 To conclude the proof of Theorem \ref{thm:complete}, we return to the process \eqref{eq:generator-1}.
 The proof is based on the simple observation that processes \eqref{eq:generator-1} and \eqref{eq:generator-2} are equal as long as no
 local event occurs in any of the vertices not in state 0.
 The bound on the number of occupied patches and the bound on the time to dispersion given respectively by Lemmas \ref{lem:leaves}
 and \ref{lem:dispersion} show that the time to dispersion for the process that includes local events is small as well
 with probability close to one, so the result follows from Lemma \ref{lem:upper-lower}.
 This argument is made rigorous in the following lemma.
\begin{lemma} --
\label{lem:extinction}
 For all $N$ sufficiently large, we have
\begin{equation}
\label{eq:extinction-1}
  p_N^+ (\theta, \mu) \ \leq \ 2^{n + 1} \,(\ln N)^{-1} \ + \ 3 \,(\ln N)^{-a} \ + \ n \,(\ln N)^{2n + 1} / N.
\end{equation}
\end{lemma}
\begin{proof}
 The two processes \eqref{eq:generator-1} and \eqref{eq:generator-2} being constructed from the same graphical representation
 are equal as long as no local event occurs in any of the vertices not in state 0.
 Since local events occur at each vertex at rate one, having a collection $\zeta_1, \zeta_2, \zeta_3, \ldots$ of independent
 exponential random variables with parameter one, we deduce that
\begin{equation}
\label{eq:extinction-2}
  \begin{array}{l}
    P \,(\eta_t \not \equiv \xi_t \ \hbox{for some} \ t \leq T_N \,| \,|\xi_{T_N}| \leq K_N) \vspace*{4pt} \\ \hspace*{20pt} \leq \
    P \,(\min \,\{\zeta_i : i \leq K_N \} < T_N) \ = \
    1 - P \,(\min \,\{\zeta_i : i \leq K_N \} \geq T_N) \vspace*{4pt} \\ \hspace*{20pt} \leq \
    1 - \exp \,(- K_N T_N) \ \leq \ 1 - \exp \,(- n \,(\ln N)^{2n} \,\ln (\ln N) / N) \ \leq \ n \,(\ln N)^{2n + 1} / N \end{array}
\end{equation}
 for all $N$ large.
 From Lemmas \ref{lem:leaves} and \ref{lem:dispersion}, and \eqref{eq:extinction-2}, we obtain
 $$ \begin{array}{l}
     P \,(\eta_{T_N} \notin \Omega_-) \ \leq \
     P \,(\xi_{T_N} \notin \Omega_-) \ + \ P \,(\eta_t \not \equiv \xi_t \ \hbox{for some} \ t \leq T_N) \vspace*{4pt} \\ \hspace*{35pt} \leq \
     P \,(\tau_D > T_N) \ + \ P \,(\eta_t \not \equiv \xi_t \ \hbox{for some} \ t \leq T_N \,| \,|\xi_{T_N}| \leq K_N) \ + \ P \,(|\xi_{T_N}| > K_N) \vspace*{4pt} \\ \hspace*{70pt} \leq \
     2^{n + 1} \,(\ln N)^{-1} \ + \ 3 \,(\ln N)^{-a} \ + \ n \,(\ln N)^{2n + 1} / N \end{array} $$
 for all $N$ large.
 Since according to Lemma \ref{lem:upper-lower} we have
 $$ p_N^+ (\theta, \mu) \ = \ P \,(\eta_t \in \Omega_+ \ \hbox{for some} \ t) \ = \ P \,(\eta_t \notin \Omega_- \ \hbox{for all} \ t) \ \leq \ P \,(\eta_{T_N} \notin \Omega_-) $$
 the proof is complete.
\end{proof} \\ \\
 The theorem directly follows from Lemma \ref{lem:extinction} by observing that the right-hand side of \eqref{eq:extinction-1}
 tends to zero as the number of vertices $N$ goes to infinity.


\section{The process on the ring}
\label{sec:ring}

\indent This section is devoted to the proof of Theorem \ref{thm:ring}.
 To understand the process on the ring, the first step is to study its counterpart on the infinite one-dimensional lattice
 using a so-called block construction.
 The idea is to couple a certain collection of good events related to the infinite system with the set of open sites of
 a one dependent oriented site percolation process on
 $$ \mathcal H \ := \ \{(x, n) \in \Z \times \Z_+ : x + n \ \hbox{is even} \}. $$
 For a precise definition and a review of oriented site percolation in two dimensions, we refer the reader
 to Durrett \cite{durrett_1984}.
 This coupling together with results from \cite{durrett_1995} implies that, starting with a single occupied patch, there exists
 with positive probability a linearly expanding region that contains a positive density of patches above the Allee threshold.
 The second key step is to prove that in fact all the patches in this space-time region are above the Allee threshold, from
 which it follows that, with the same positive probability, the process on the finite ring starting with a single occupied patch
 hits an upper configuration before it hits a lower configuration.
 To prove linear expansion in space of the set of patches that exceed the Allee threshold, we observe that, under the assumptions
 of the theorem, there exists a constant $a$ fixed from now on such that
 $$ a \ < \ (1 - \mu)^{4T} \quad \hbox{and} \quad a \,\mu^2 \,(1 - \mu)^{8T} \ > \ \theta \quad \hbox{where} \quad T := 95. $$
 To define our collection of good events, we also introduce random variables that keep track of the number of mixing events and
 local events in certain space-time regions of the graphical representation.
 More precisely, we introduce the number of mixing events
 $$ \begin{array}{rcl}
      X_j & := & \card \{n : T_n (j, j + 1) \in (0, T) \} \quad \hbox{for} \ j = -1, -2 \vspace*{4pt} \\
      X_j & := & \card \{n : T_n (j - 1, j) \in (0, T) \} \quad \hbox{for} \ j = +1, +2 \vspace*{4pt} \\
      Y_j & := & \card \{n : T_n (j - 1, j) \in (T, 2T) \} \vspace*{4pt} \\ && \hspace{40pt} + \
                 \card \{n : T_n (j, j + 1) \in (T, 2T) \} \quad \hbox{for} \ j = -1, +1 \end{array} $$
 as well as the number of local events
 $$ Z_j \ := \ \card \{n : U_n (j) \in (T, 2T) \} \quad \hbox{for} \ j = -1, +1. $$
 From these random variables, we define the good event
 $$ \Omega \ := \ \{\min \,(X_{-1}, X_1, Z_{-1}, Z_1) \neq 0 \} \,\cap
                \,\{\max \,(X_{-2}, X_2) \leq 2T \} \,\cap
                \,\{\max \,(Y_{-1}, Y_1) \leq 4T \} $$
 as depicted in Figure \ref{fig:perco}.
 For every $(x, n) \in \mathcal H$, we define the good event $\Omega (x, n)$ similarly but from the graphical representation of
 the process in the space-time region
 $$ R (x, n) \ := \ (x - 2, x + 2) \times (2nT, 2nT + 2T). $$
 The motivation for introducing these events is that, conditioned on $\Omega (x, n)$, if the population density at patch $x$ at
 time $2nT$ exceeds $a$ then the same holds for the two adjacent patches $2T$ units of time later.
 By translation invariance of the evolution rules of the process in space and time, it suffices to prove the result for
 $x = n = 0$, which is done in the following lemma.

\begin{figure}[t]
\centering
\scalebox{0.40}{\input{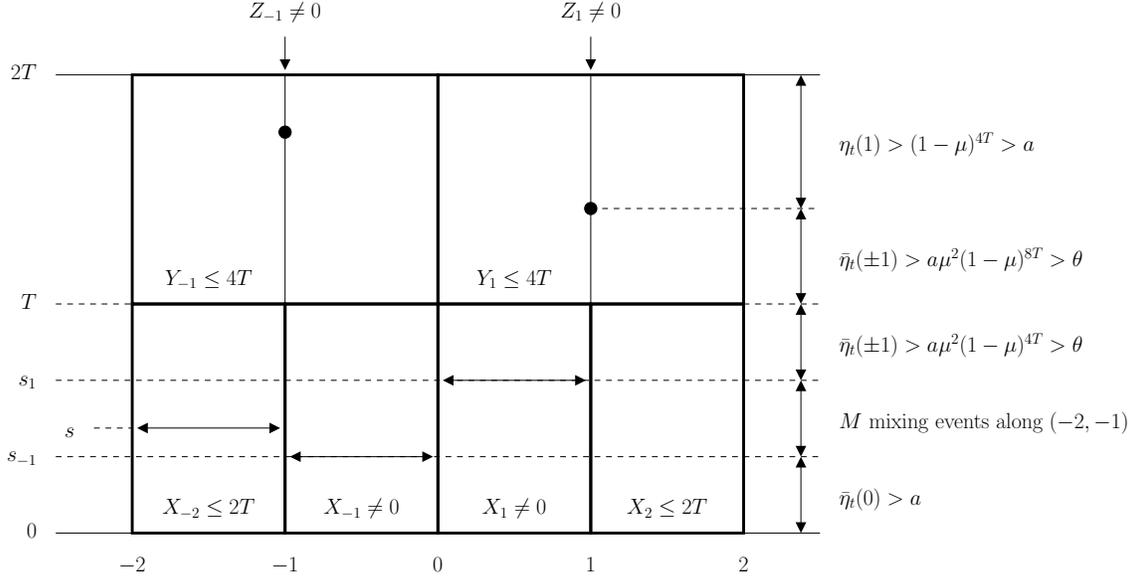}}
\caption{\upshape{Picture of the good event $\Omega = \Omega (0 ,0)$.}}
\label{fig:perco}
\end{figure}

\begin{lemma} --
\label{lem:invasion}
 We have $\{\eta_0 (0) > a \} \cap \Omega (0, 0) \subset \{\min \,(\eta_{2T} (- 1), \eta_{2T} (1)) > a \}$.
\end{lemma}
\begin{proof}
 The first step is to prove that the event on the left-hand side is included in the event that the population density at $-1$
 and 1 exceed the Allee threshold at time $2T$ for the process \eqref{eq:generator-2} that includes mixing events but excludes
 local events.
 Assume that
 $$ s_{-1} \ := \ T_1 (-1, 0) \ < \ T_1 (0, 1) \ =: \ s_1 \ < \ T. $$
 Note that the last inequality $s_1 < T$ follows from $X_1 \neq 0$.
 To study the process up until time $s_1$, we introduce the number of mixing events
 $$ M \ := \ \card \{n : T_n (-2, -1) \in (s_{-1}, s_1) \}, $$
 and in case $M \neq 0$, the time of the first mixing event
 $$ s \ := \ \inf \,(\{T_n (-2, -1) : n \geq 1 \} \,\cap \,(s_{-1}, s_1)). $$
 Since the population density at patch $-1$ and at patch 0 between time $s_{-1}$ and time $s$ are convex combinations
 of their counterpart at time $s_{-1}$, we have
 $$ \min \,(\xi_t (-1), \xi_t (0)) \ \geq \ \min \,(\xi_{s_{-1}} (-1), \xi_{s_{-1}} (0)) \ > \ \min \,(a \mu, a (1 - \mu)) \ = \ a \mu $$
 for all $t \in (s_{-1}, s)$. In particular,
 $$ \min \,(\xi_s (-1), \xi_s (0)) \ > \ a \,\mu \,(1 - \mu). $$
 Since $M \leq X_{-2} \leq 2T$, we deduce from a simple induction that
\begin{equation}
\label{eq:invasion-1}
  \min \,(\xi_t (-1), \xi_t (0)) \ > \ a \,\mu \,(1 - \mu)^M \ \geq \ a \,\mu \,(1 - \mu)^{2T} \ > \ \theta
\end{equation}
 for all $t \in (s_{-1}, s_1)$, and by definition of $s_1$,
 $$ \xi_{s_1} (1) \ = \ \mu \,\xi_{s_1 -} (0) + (1 - \mu) \,\xi_{s_1 -} (1) \ \geq \ \mu \,\xi_{s_1 -} (0) \ > \ a \,\mu^2 \,(1 - \mu)^M. $$
 In particular, using again our inductive reasoning and the fact that
 $$ \begin{array}{l}
    \card \{n : T_n (- 2, - 1) \in (s_1, T) \} \vspace*{4pt} \\ \hspace{50pt} + \
    \card \{n : T_n (1, 2) \in (s_1, T) \} \ \leq \ X_{-2} + X_2 - M \ \leq \ 4T - M \end{array} $$
 we deduce that
\begin{equation}
\label{eq:invasion-2}
 \begin{array}{rcl}
 \min \,(\xi_t (-1), \xi_t (0), \xi_t (1)) & \geq &
        (1 - \mu)^{4T - M} \ \min \,(\xi_{s_1} (-1), \xi_{s_1} (0), \xi_{s_1} (1)) \vspace*{4pt} \\ & > &
        (1 - \mu)^{4T - M} \ a \,\mu^2 \,(1 - \mu)^M \ = \ a \,\mu^2 \,(1 - \mu)^{4T} \ > \ \theta \end{array}
\end{equation}
 for all $t \in (s_1, T)$.
 Finally, since $Y_{-1} \leq 4T$ and $Y_1 \leq 4T$,
\begin{equation}
\label{eq:invasion-3}
 \begin{array}{rcl}
 \min \,(\xi_t (-1), \xi_t (1)) & \geq &
        (1 - \mu)^{4T} \ \min \,(\xi_T (-1), \xi_T (1)) \vspace*{4pt} \\ & > &
        (1 - \mu)^{4T} \ a \,\mu^2 \,(1 - \mu)^{4T} \ = \ a \,\mu^2 \,(1 - \mu)^{8T} \ > \ \theta \end{array}
\end{equation}
 for all $t \in (T, 2T)$.
 Returning to the system with local events, since
 $$ \xi_t (0) > \theta \ \ \hbox{for all} \ t \in (0, T) \quad \hbox{and} \quad \xi_t (\pm 1) > \theta \ \ \hbox{for all} \ t \in (s_{\pm 1}, 2T) $$
 according to \eqref{eq:invasion-1}--\eqref{eq:invasion-3}, these inequalities remain true for the original process \eqref{eq:generator-1}.
 In particular, the populations at patches $-1$ and 1 exceed $\theta$ between times $T$ and $2T$ therefore
 $$ \eta_t (-1) \ = \ 1 \ \ \hbox{for some} \ t \in (T, 2T) \quad \hbox{and} \quad \eta_t (1) \ = \ 1 \ \ \hbox{for some} \ t \in (T, 2T) $$
 since $Z_{-1} \neq 0$ and $Z_1 \neq 0$.
 Using again that $Y_{-1}, Y_1 \leq 4T$, we conclude that
 $$ \eta_{2T} (-1) \ \geq \ (1 - \mu)^{4T} \ > \ a \quad \hbox{and} \quad \eta_{2T} (1) \ \geq \ (1 - \mu)^{4T} \ > \ a $$
 which completes the proof of the lemma.
\end{proof} \\ \\
 To deduce from Lemma \ref{lem:invasion} the existence of a linearly expanding region with a positive density of patches above
 the Allee threshold, we now prove that the common probability of all our good events exceeds the critical value $p_c$ of
 one dependent oriented site percolation.
\begin{lemma} --
\label{lem:good-event}
 For $T = 95$, we have $P \,(\Omega (x, n)) \geq 1 - 3^{-36} > p_c$.
\end{lemma}
\begin{proof}
 The key is simply to observe that
\begin{itemize}
 \item $X_{-2}, X_{-1}, X_1$ and $X_2$ are Poisson random variables with parameter $T$, \vspace*{4pt}
 \item $Y_{-1}$ and $Y_1$ are Poisson random variables with parameter $2T$ and \vspace*{4pt}
 \item $Z_{-1}$ and $Z_1$ are Poisson random variables with parameter $T$.
\end{itemize}
 Using in addition that these random variables are independent, we deduce that
 $$ \begin{array}{rcl}
     P \,(\Omega (x, n)) & = & P \,(\Omega) \ \geq \ 1 - 4 \,P \,(X_1 = 0) - 2 \,P \,(X_2 > 2T) - 2 \,P \,(Y_1 > 4T) \vspace*{8pt} \\ & = &
       \displaystyle 1 \ - \ 4 \,e^{-T} \ - \ 2 \sum_{n > 2T} \ \frac{T^n}{n!} \ e^{-T} \ - 2 \sum_{n > 4T} \ \frac{(2T)^n}{n!} \ e^{-2T} \ > \ 1 - 3^{-36} \end{array} $$
 when $T = 95$.
 The second inequality $p_c < 1 - 3^{-36}$ in the statement follows from the contour argument described
 in, e.g., \cite{durrett_1984}, Section 10.
 This completes the proof of the lemma.
\end{proof} \\ \\
 Lemmas \ref{lem:invasion} and \ref{lem:good-event} and the fact that
 $$ R (x, n) \,\cap \,R (x', n') \ = \ \varnothing \quad \hbox{whenever} \ |x - x'| > 1 \ \hbox{or} \ n \neq n' $$
 are the assumptions of Theorem 4.3 in Durrett \cite{durrett_1995} with $M = 1$ and $\gamma = 3^{-36}$ from which it follows
 that, for the infinite system starting with a single occupied patch at the origin,
 $$ \bar W_n \ := \ \{x \in \Z : (x, n) \in \mathcal H \ \hbox{and} \ \eta_{2nT} (x) > a \} $$
 dominates stochastically the set of wet sites $W_n$ at level $n$ of a one dependent oriented site percolation process with
 parameter $1 - \gamma$ and initial condition $W_0 = \bar W_0$.
 Since $1 - \gamma > p_c$ we deduce that, with positive probability at least equal to the percolation probability, the set of
 patches that exceed the constant $a$ expands linearly.
 This only proves persistence of the metapopulation of the infinite lattice, which is not sufficient to deduce global expansion
 of the system on the ring.
 The last step is to show that all patches in the expanding region exceed the Allee threshold.
 More precisely, on the event that percolation occurs, i.e., $W_n \neq \varnothing$ for all $n$, we have
\begin{equation}
\label{eq:expand}
  \lim_{n \to \infty} l_n \ := \ \lim_{n \to \infty} \min \, W_n \ = \ - \infty \quad \hbox{and} \quad
  \lim_{n \to \infty} r_n \ := \ \lim_{n \to \infty} \max \, W_n \ = \ + \infty
\end{equation}
 and thinking of the infinite system as being coupled with one dependent oriented site percolation in such a way
 that $W_n \subset \bar W_n$ for all $n$, we have the following lemma.
\begin{lemma} --
\label{lem:expansion}
 Assume that $W_n \neq \varnothing$ for all $n$. Then, for all $x \in \Z$,
 $$ \eta_t (x) \ > \ \theta \quad \hbox{for all times $t$ sufficiently large}. $$
\end{lemma}
\begin{proof}
 Recall from the proof of Theorem 4.3 in Durrett \cite{durrett_1995} that the processes are coupled in such a way that the set of open
 sites for the percolation process is included in the set of good sites, where site $(x, n)$ is said to be good whenever the
 good event $\Omega (x, n)$ occurs.
 This, together with the definition of the right edge, implies that there is a good path from site $(0, 0)$ to site $(r_n, n)$,
 i.e., a sequence of integers $x_0 = 0, x_1, \ldots, x_n = r_n$ such that
 $$ (x_m, m) \ \hbox{is good for} \ m = 0, 1, \ldots, n \quad \hbox{and} \quad |x_m - x_{m - 1}| = 1 \ \hbox{for} \ m = 1, 2, \ldots, n. $$
 Since $\eta_0 (0) > a$, it follows from \eqref{eq:invasion-1}--\eqref{eq:invasion-3} in the proof of Lemma \ref{lem:invasion} that,
 for $m = 0, 1, \ldots, n$,
\begin{equation}
\label{eq:right}
  \begin{array}{rcl}
   \eta_t (x_m) \ > \ \theta & \hbox{for all} & t \in [2mT, 2mT + T) \vspace*{4pt} \\
   \eta_t (x_m + 1) \ > \ \theta & \hbox{for all} & t \in [2mT + T, 2mT + 2T). \end{array}
\end{equation}
 Similarly, there is a path $(x_0, 0) \to (x_{-1}, 1) \to \cdots \to (x_{-n}, n) = (l_n, n)$ such that
\begin{equation}
\label{eq:left}
  \begin{array}{rcl}
   \eta_t (x_{-m}) \ > \ \theta & \hbox{for all} & t \in [2mT, 2mT + T) \vspace*{4pt} \\
   \eta_t (x_{-m} - 1) \ > \ \theta & \hbox{for all} & t \in [2mT + T, 2mT + 2T) \end{array}
\end{equation}
 for $m = 0, 1, \ldots, n$, and we may assume that $x_{-m} \leq x_m$ for all $m$.
 We claim that all patches in the space-time region delimited by \eqref{eq:right} and \eqref{eq:left} are above the
 Allee threshold $\theta$, i.e.,
\begin{equation}
\label{eq:right-left}
  \begin{array}{rcl}
   \eta_t (x) \ > \ \theta & \hbox{for all} & (x, t) \in [x_{-m}, x_m] \times [2mT, 2mT + T) \vspace*{4pt} \\
   \eta_t (x) \ > \ \theta & \hbox{for all} & (x, t) \in [x_{-m} - 1, x_m + 1] \times [2mT + T, 2mT + 2T), \end{array}
\end{equation}
 which we prove by induction.
 Assume that \eqref{eq:right-left} holds for some $m < n$.
 The fact that this again holds at time $t = 2mT + 2T$ simply follows from the fact that
 $$ [x_{- m - 1}, x_{m + 1}] \ \subset \ [x_{-m} - 1, x_m + 1] \quad \hbox{since} \quad x_{- m - 1} \geq x_{-m} - 1 \ \ \hbox{and} \ \ x_{m + 1} \leq x_m + 1. $$
 To prove that this holds at later times, we distinguish three types of events:
\begin{itemize}
 \item Local events cannot violate the first line of \eqref{eq:right-left} since patches above the Allee threshold can only
  experience local expansions to their carrying capacity. \vspace*{4pt}
 \item Mixing events in $[x_{- m - 1}, x_{m + 1}]$ cannot violate the first line of \eqref{eq:right-left} since the
  new states of interacting patches above the Allee threshold are again above the Allee threshold. \vspace*{4pt}
 \item Mixing events along $(x_{- m - 1} - 1, x_{- m - 1})$ or $(x_{m + 1}, x_{m + 1} + 1)$ can violate the first line
  of \eqref{eq:right-left} but this would contradict either \eqref{eq:right} or \eqref{eq:left}.
\end{itemize}
 This proves the first line of \eqref{eq:right-left} at step $m + 1$ while the second line follows from the exact same reasoning.
 From \eqref{eq:right-left}, we deduce that
 $$ \eta_t (x) \ > \ \theta \quad \hbox{for all} \quad (x, t) \in [l_n, r_n] \times [2mT, 2mT + 2T). $$
 In particular, the lemma follows from \eqref{eq:expand}.
\end{proof} \\ \\
 Returning to the process on the ring, we deduce from Lemma \ref{lem:expansion} that, with positive probability at least
 equal to the percolation probability and starting with a single patch in state~1, the system reaches an upper configuration
 before it reaches a lower configuration, an event that leads to global expansion according to Lemma \ref{lem:upper-lower}.
 This completes the proof of Theorem \ref{thm:ring}.





\begin{thebibliography}{10}

\bibitem{durrett_1984}
 Durrett, R. (1984). Oriented percolation in two dimensions.
\emph{Ann. Probab.} \textbf{12} 999--1040.

\bibitem{durrett_1995}
 Durrett, R. (1995). Ten lectures on particle systems.
 In \emph{Lectures on probability theory (Saint-Flour, 1993)}, volume 1608 of \emph{Lecture Notes in Math.}, pages 97--201.
 Springer, Berlin.

\bibitem{harris_1972}
 Harris, T. E. (1972). Nearest neighbor Markov interaction processes on multidimensional lattices.
\emph{Adv. Math.} \textbf{9} 66--89.

\bibitem{kang_lanchier_2011}
 Kang, Y. and Lanchier, N. (2011). Expansion or extinction: deterministic and stochastic two-patch models with Allee effects.
\emph{J. Math. Biol.} \textbf{62} 925--973.

\end{thebibliography}
\end{document}